\documentclass[a4 paper,11pt]{article}
\usepackage{stmaryrd}
\usepackage{amsfonts}
\usepackage{bbm}
\usepackage{amscd}
\usepackage{mathrsfs}
\usepackage{latexsym,amssymb,amsmath,amscd,amscd,amsthm,amsxtra,xypic}
\usepackage[dvips]{graphicx}
\usepackage[utf8]{inputenc}
\usepackage[T1]{fontenc}
\usepackage{enumerate}
\usepackage{lmodern}
\usepackage{amssymb}
\usepackage[all]{xy}
\usepackage{ifpdf}
\ifpdf
\usepackage[colorlinks=true,linkcolor=blue,final,backref=page,hyperindex,citecolor=red]{hyperref}
\else
\usepackage[colorlinks,final,backref=page,hyperindex,hypertex]{hyperref}
\fi

\usepackage{nicefrac,mathtools,enumitem}
\usepackage{microtype}
\usepackage{amsfonts}
\usepackage{amssymb}
\usepackage{mathrsfs}
\usepackage{amsmath}
\usepackage{amsthm}
\usepackage{enumerate}
\usepackage{indentfirst}
\usepackage{amsfonts}
\usepackage{amssymb}
\usepackage{mathrsfs}
\usepackage{amsmath}
\usepackage{amsthm}
\usepackage{enumerate}
\usepackage{cite}
\usepackage{mathrsfs}
\usepackage{geometry}
\allowdisplaybreaks

\newtheorem{thm}{Theorem}[section]

\newtheorem{pro}[thm]{Proposition}

\newtheorem{rmk}[thm]{Remark}
\newtheorem{defi}[thm]{Definition}

\numberwithin{equation}{section}

\newcommand{\be }{\begin{equation}}
\newcommand{\ee }{\end{equation}}

\newcommand{\br}[1]{   [ \cdot,    \cdot  ]   }

\newcommand {\emptycomment}[1]{}

\def\<{\langle}
\def\>{\rangle}

\def\c{\cdot}

\begin{document}

\title{\bf  The solenoidal Virasoro algebra and its simple weight modules}

\author{\bf B. Agrebaoui, W. Mhiri}
\author{{ Boujemaa Agrebaoui$^{1}$
 \footnote{E-mail: b.agreba@fss.rnu.tn}~ and \  Walid Mhiri$^{1}$ \footnote{E-mail: mhiriw1@gmail.com}
\
}\\
 \\
{\small 1.~University of Sfax, Faculty of Sciences Sfax,  BP
1171, 3038 Sfax, Tunisia}}

\maketitle

\begin{abstract}
Let $A_n=\mathbb{C}[t_i^{\pm1},~1\leq i\leq n]$ be the algebra of Laurent polynomials in $n$-variables.

 Let $\mu=(\mu_1,\ldots,\mu_n)$ be a generic vector in $\mathbb{C}^n$ and $\Gamma_{\mu}=\{\mu\cdot\alpha,\alpha\in \mathbb{Z}^n\}$ where
  $\mu\cdot\alpha=\displaystyle\sum_{i=1}^n\mu_i\alpha_i$ for $\alpha=(\alpha_1,\ldots,\alpha_n)\in \mathbb{Z}^n$. Denote by $d_\mu$ the vector field:
 $$d_\mu=\displaystyle\sum_{i=1}^n\mu_it_i\frac{d}{dt_i}.$$ In \cite{BiFu}, Y. Billig and V. Futorny introduce the solenoidal Lie algebra $\mathbf{W}(n)_{\mu}:=A_nd_\mu$, where the Lie structure is given by the commutators of vector fields.

In the first part of this paper, we study the universal central extension  of $\mathbf{W}(n)_{\mu}$. We obtain a rank $n$ Virasoro algebra  called the solenoidal Virasoro algebra $\mathbf{Vir}(n)_\mu$.

In the second part, we recall in the case of $\mathbf{Vir}(n)_\mu$, the well know Harich-Chandra modules for generalized Virasoro algebra studied in \cite{Su,Su1,LuZhao}.

In the third part, we construct irreducible highest and lowest $\mathbf{Vir}(n)_\mu$-modules using triangular decomposition given by lexicographic order on $\mathbb{Z}^{n}$. We prove that these modules are weight modules which have infinite dimensional weight spaces.
\end{abstract}

\textbf{Key words}: Virasoro algebra, solenoidal algebra, central extension, Harish-Chandra modules, cuspidal modules

\textbf{Mathematics Subject Classification} (2020): 17B10,17B20,17B68,17B86.

\numberwithin{equation}{section}

\tableofcontents

\section{Introduction}
Let $A_1=\mathbb{C}[t,t^{-1}]$ the algebra of Laurent polynomials in one variable. The Witt algebra  $\mathbf{W}(1):=Der(A_1)$ is the algebra of all derivations of the ring $A_1$ known also as the centerless Virasoro algebra. Let $e_i=t^{i+1}\frac{d}{dt},~i\in \mathbb{Z}$, be the canonical basis of $\mathbf{W}(1)$ with the Lie bracket generated by $[e_i,e_j]=(j-i)e_{i+j}$.
The universal central extension of $\mathbf{W}(1)$ is the well known Virasoro algebra $\mathbf{Vir}:=\mathbf{W}(1)\oplus\mathbb{C}C$ where the Lie bracket is given by: $$[e_i,e_j]=(j-i)e_{i+j}+\frac{i^3-i}{12}\delta_{i+j,0}C.$$

In 1981, V. Kac  \cite{K1,K2} consider some problems on infinite dimensional Lie algebras and their representations. In particular he study the representations of the Virasoro algebra $\mathbf{Vir}$ with finite dimensional weight spaces knows as Harish-Chandra modules. He conjectured that these modules  are either the highest weight modules, the lowest weight modules or the intermediate series modules known also as modules of tensor densities. In 1992, O. Mathieu \cite{Ma} using Chang's results with elegant fashion, he proves the Kac's conjecture.

Patera and Zassenhaus  in \cite{PZ} introduced the generalized Virasoro algebra $\mathbf{Vir}[M]$ for any
additive subgroup $M$ of $\mathbb{C}$. This Lie algebra can be obtained from $\mathbf{Vir}$ by replacing the index
group $\mathbb{Z}$ with $M$. If $M\simeq\mathbb{Z}^n$, then $\mathbf{Vir}[M]$ is called a rank $n$ Virasoro
algebra (or a higher rank Virasoro algebra if $n\geq 2$).
Representations for generalized Virasoro algebras $\mathbf{Vir}[M]$ have been studied by several
authors. Mazorchuk \cite{M} proved that all irreducible weight modules with finite dimensional
weight spaces over $\mathbf{Vir}[\mathbb{Q}]$ are intermediate series modules (where $\mathbb{Q}$ is the field of rational
numbers). In \cite{M1}, Mazorchuk determined the irreducibility of Verma modules with zero
central charge over higher rank Virasoro algebras. In \cite{HuWaZh}, Hu, Wang and Zhao obtained
a criterion for the irreducibility of Verma modules over the generalized Virasoro algebra
$\mathbf{Vir}[M]$ over an arbitrary field $\mathbb{F}$ of characteristic $0$ ($M$ is an additive subgroup of $\mathbb{F}$).
In \cite{SuZhao}, Y. Su and K. Zhao proved that irreducible weight modules with weight spaces of bounded dimension are
the intermediate series. In \cite{Su, Su1} Y.Su proved that the irreducible
Harish-Chandra modules over higher rank Virasoro algebras are divided into two classes: intermediate series modules, and \textbf{GHW} modules. In \cite{BeBill}, Y. Billig and K. Zhao constructed a new class of irreducible weight modules with finite dimensional weight spaces over some generalized
Virasoro algebras. In \cite{LuZhao}  R. Lu and K. Zhao prove that irreducible weight modules which are  not uniformly bounded in the classification in \cite{BeBill} are exactly \textbf{GHW} modules introduced in \cite{Su,Su1}.

Let $A_n=\mathbb{C}[t^{\pm 1}_1,\ldots t^{\pm 1}_n]$ be the algebra of Laurent polynomials in $n$ variables and let $\mathbf{W}(n):=Der(A_n)$ the Lie algebra of derivation on $A_n$ known also as the Lie algebra of polynomial vector fields on the torus or the rank $n$ Witt algebra. The Lie bracket of $\mathbf{W}(n)$ is given by commutators of vector fields. Recently in \cite{BiFu1}, Y. Billig and V. Futorny studied the Harish-Chandra modules of $\mathbf{W}(n)$.

Let $d_{\mu}:=\sum_{i=1}^n\mu_it_i\frac{\partial}{\partial t_i}$ where $\mu=(\mu_1,\ldots,\mu_n)\in \mathbb{C}^n$ is generic. In \cite{BiFu}, Y. Billig and V. Futorny introduced the so called solenoidal Lie algebra $\mathbf{W}(n)_\mu:=A_nd_{\mu}$ as Lie subalgebra of $\mathbf{W}(n)$.
Let $\Gamma_{\mu}:=\{\mu\cdot\alpha,~\alpha\in \mathbb{Z}^n\}$. It  is a lattice of  $(\mathbb{C},+)$ isomorphic to $\mathbb{Z}^n$.
The algebra $\mathbf{W}(n)_\mu$ is $\Gamma_{\mu}$-graded with respect to the CSA $\mathbb{C}d_{\mu}$.

In the first part of the present paper, we study the universal central extensions of $\mathbf{W}(n)_\mu$. We obtain a generalized Virasoro algebra. We denote this algebra by $\mathbf{Vir}(n)_\mu$ and should call it the solenoidal Virasoro algebra.  The main result of this paper is  Theorem \ref{Ext1}.

In the second part, we study Harish-Chandra modules of $\mathbf{Vir}(n)_\mu$.
First, we extend cuspidal modules $T_{\mu}(a,b)$ of $\mathbf{W}(n)_\mu$ studied by Y. Billig and V. Futorny in \cite{BiFu} to $\mathbf{Vir}(n)_\mu$-modules by letting the central charge acting by zero. Then following \cite{LuZhao}, we construct generalized Verma modules $\widetilde{V}(a,b,\Gamma_{\mu'})$ and we take their irreducible quotients $\overline{V}(a,b,\Gamma_{\mu'})$. As proved in \cite{LuZhao}, the modules $\overline{V}(a,b,\Gamma_{\mu'})$ are isomorphic to generalized highest weight modules (\textbf{GHW} modules) introduced and study in \cite{Su,Su1} . We end this part by  Theorem \ref{thmS} which is a classification of Harish-Chandra modules of $\mathbf{Vir}(n)_\mu$. It is a particular case of results in \cite{BeBill, LuZhao,Su,Su1} for generalized Virasoro algebras.

In the third part, we consider the order $\prec$ on $\Gamma_\mu$ induced by the lexicographic order $<_{lex}$ on $\mathbb{Z}^{n}$ as follows:
$$\mu.\alpha\prec \mu.\beta \hbox{ if and only if } \alpha<_{lex} \beta.$$
Let $\Gamma_\mu^+:=\{\mu\cdot\alpha| \overrightarrow{0}<_{lex}\alpha\}$ and $\Gamma_\mu^-:=\{\mu\cdot\alpha| \alpha<_{lex}\overrightarrow{0}\}$. Then $\mathbf{Vir}(n)_{\mu}$ has triangular decomposition:
$$\mathbf{Vir}(n)_{\mu}=(\mathbf{Vir}(n)_{\mu})_+\oplus(\mathbf{Vir}(n)_{\mu})_0\oplus (\mathbf{Vir}(n)_{\mu})_{-}$$
where, $(\mathbf{Vir}(n)_{\mu})_0=\mathbb{C}d_\mu\oplus \mathbb{C}c_\mu$ is the Cartan subalgebra and $(\mathbf{Vir}(n)_{\mu})_{\pm}=\displaystyle \oplus_{\mu.\alpha\in\Gamma_\mu^{\pm} }\mathbb{C}e_{\mu.\alpha}.$

We consider the Verma module associated to this triangular decomposition:
$$M(\lambda,c) = Ind^{\mathbf{Vir}(n)_\mu}_{\mathfrak{b}^+}\mathbb{C}_{\lambda,c}$$
where $\mathfrak{b}^+=(\mathbf{Vir}(n)_{\mu})_+\oplus(\mathbf{Vir}(n)_{\mu})_0$ and $\mathbb{C}_{\lambda,c}$ is a one dimensional $\mathfrak{b}^+$-module such that $(\mathbf{Vir}(n)_{\mu})_+$ acts by zero, $d_\mu$ acts by a multiplication by $\lambda$ and $c_\mu$ by a multiplication by $c$.

The Verma module $M(\lambda,c)$ has a maximal proper submodule $\overline{M(\lambda,c)}$ and the quotient $V(\lambda,c):=M(\lambda,c)/ \overline{M(\lambda,c)}$ is an irreducible highest weight module of $\mathbf{Vir}(n)_{\mu}$. We construct irreducible lowest weight modules with the same manner. These modules have infinite dimensional weight spaces.

This paper is arranged as follows.\textbf{ In Section 2}, we compute the universal central extension of the solenoidal Lie algebra $\mathbf{W}(n)_{\mu}$
  and we  introduce the solenoidal-Virasoro algebra $\mathbf{Vir}(n)_\mu$. \textbf{In section 3}, we study Harish-Chandra $\mathbf{Vir}(n)_\mu$-modules. \textbf{The section 4} is devoted to  irreducible weight modules having infinite dimensional weight spaces.

\section{The solenoidal Virasoro algebra}
Let $A_n=\mathbb{C}[t_{i}^{\pm1},~1\leq i\leq n]$ be the algebra of Laurent polynomials. For $\alpha=(\alpha_1,\ldots,\alpha_n)\in \mathbb{Z}^n$ and $\mu=(\mu_{1},\ldots,\mu_{n})\in\mathbb{C}^{n}$, let $|\alpha|:=\displaystyle\sum_{i=1}^n\alpha_i$ and $\mu\cdot\alpha:=\displaystyle\sum_{i=1}^n\mu_i\alpha_i$.
 Assume that $\mu$ is a generic vector that is $\mu\cdot\alpha\neq 0,\forall~ \alpha\in \mathbb{Z}^n\setminus\{\overrightarrow{0}\}$.

 Let  $\Gamma_{\mu}$ be the lattice of  $\mathbb{C}$ image of $\mathbb{Z}^{n}$ by the embedding map: $$\begin{array}{cc}
                                                                     \sigma_\mu:& \mathbb{Z}^{n}\longrightarrow  \mathbb{C} \\
                                                                      &\alpha\mapsto  \mu\cdot\alpha
                                                                      \end{array}$$
 Let $d_{\mu}:=\sum_{i=1}^{n}\mu_{i}D_{t_{i}},$ where $D_{t_{i}}=t_{i}\frac{\partial}{\partial t_{i}}$. The subalgebra  $\mathbf{W}(n)_{\mu}:=A_nd_{\mu}$
  of $\mathbf{W}(n)$ is introduced by Y. Billig and V. Futorny in \cite{BiFu} and is called the  solenoidal-Witt algebra.
  The canonical basis of $\mathbf{W}(n)_{\mu}$ is given by: $$\{e_{\mu\cdot\alpha}:=t^{\alpha}d_{\mu},\mu\cdot\alpha\in\Gamma_{\mu}\}.$$
 Its Lie algebra structure is generated by the commutators:
\begin{equation}\label{soloWitt}[e_{\mu\cdot\alpha},e_{\mu\cdot\beta}]=\mu\cdot(\beta-\alpha)e_{\mu\cdot(\alpha+\beta)},~~\mu\cdot\alpha,\mu\cdot\beta\in \Gamma_{\mu}.\end{equation}

If $n=1$,  $\Gamma_{\mu}=\mu\mathbb{Z}$ for some $\mu \in \mathbb{C}^*$. In this case the algebra $\mathbf{W}(n)_{\mu}$ is isomorphic to  the classical Witt algebra  $\mathbf{W}(1)$ associated to $\mu =1$.

In the following, we will compute the universal central extension of the solenoidal Lie algebra $\mathbf{W}(n)_{\mu}$. In the case $n=1$, the central extension is the well known Virasoro algebra (see for example the book of Kac and Raina, \cite{KaRe}).
\begin{thm}\label{Ext1} The solenoidal-Witt algebra $\mathbf{W}(n)_{\mu}$ has a one-dimensional universal central extension. It is generated by the $2$-cocycle:
$$\begin{array}{ll}
C_{\mu}:& \mathbf{W}(n)_{\mu}\times \mathbf{W}(n)_{\mu}\rightarrow \mathbb{C}\\&(e_{\mu\cdot\alpha},e_{\mu\cdot\beta})\mapsto\frac{(\mu\cdot\alpha)^{3}-\mu\cdot\alpha}{12}\delta_{\alpha,-\beta}c_\mu.
\end{array}.$$

The central extension  $\mathbf{Vir}(n)_{\mu}:=\mathbf{W}(n)_{\mu}\oplus \mathbb{C}c_{\mu}$, is called the solenoidal Virasoro algebra.

Its Lie structure
in the basis $\{e_{\mu\cdot\alpha}=t^{\alpha}d_{\mu},c_{\mu}|\mu\cdot\alpha\in\Gamma_{\mu}\}$, is generated by the following brackets:
\begin{equation} \label{seloVir}
\begin{array}{cc}
\displaystyle[e_{\mu\cdot\alpha},e_{\mu\cdot\beta}]_{SVir}=\mu\cdot(\beta-\alpha)e_{\mu\cdot(\alpha+\beta)}+\frac{(\mu\cdot\alpha)^{3}-\mu\cdot\alpha}{12}\delta_{\alpha,-\beta}c_{\mu},\\[8pt]
\displaystyle[\mathbf{Vir}(n)_{\mu},c_{\mu}]_{SVir}=0.\end{array}\end{equation}
\end{thm}

\begin{proof}: The proof is divided in two part the existence and the unicity.

1) Let us first check  that the bilinear map: $C_{\mu}:\mathbf{W}(n)_{\mu}\times \mathbf{W}(n)_{\mu}\rightarrow \mathbb{C}c_\mu$ defined on the basis of $\mathbf{W}(n)_\mu$ by:$$C_{\mu}(e_{\mu\cdot\alpha},e_{\mu\cdot\beta})=\frac{(\mu\cdot\alpha)^{3}-(\mu\cdot\alpha)}{12}\delta_{\alpha,-\beta}c_\mu$$
satisfies the following $2$-cocycle condition:

\begin{equation}\label{CC} C_{\mu}(e_{\mu\cdot\alpha},[e_{\mu\cdot\kappa},e_{\mu\cdot\beta}])+C_{\mu}(e_{\mu\cdot\beta},[e_{\mu\cdot\alpha},e_{\mu\cdot\kappa}])+C_{\mu}(e_{\mu\cdot\kappa},[e_{\mu\cdot\beta},e_{\mu\cdot\alpha}])=0
.\end{equation}
 Let $\varphi_\mu(\alpha,\beta ,\kappa):=\mu\cdot(\beta-\kappa)\frac{(\mu\cdot\alpha)^3-\mu\cdot\alpha}{12}.$
 Then the right member of (\ref{CC}) becomes:
\begin{equation} \label{ECC} \Big(\varphi_\mu(\alpha,\beta ,\kappa)+\varphi_\mu(\beta ,\kappa,\alpha)+\varphi_\mu(\kappa,\alpha,\beta)\Big)\delta_{\alpha+\beta+\kappa,\overrightarrow{0}}c_\mu.\end{equation}

If  $\alpha+\kappa+\beta\neq \overrightarrow{0}$ then (\ref{ECC}) will be zero and the $2$-cocycle condition (\ref{CC}) is always satisfied.

Now, if $\alpha+\kappa+\beta=\overrightarrow{0}$ then $\delta_{\alpha+\beta+\kappa,\overrightarrow{0}}=1$.
If we replace $\kappa$ by $-\alpha-\beta$ in (\ref{ECC}) we obtain a polynomial equation in $\mu\cdot\alpha$ and $\mu\cdot\beta$. Now if we develop the expression using the linearity of $\alpha\mapsto\mu\cdot\alpha$, we will obtain zero and again the $2$-cocycle condition (\ref{CC}) is satisfied.

2) Now, let us prove the unicity.Assume that for $\mu\cdot\alpha,~\mu\cdot\beta\in \Gamma_\mu$, there exists $\theta(\mu\cdot\alpha,\mu\cdot\beta)\in \mathbb{C}$
such that:
\begin{equation}\label{bracketsolVir}[e_{\mu\cdot\alpha},e_{\mu\cdot\beta}]_{SVir}=\mu\cdot(\beta-\alpha)e_{\mu\cdot(\alpha+\beta)}+\theta(\mu\cdot\alpha,\mu\cdot\beta)c_\mu,[e_{\mu\cdot\alpha},c_\mu]_{SVir}=0\end{equation}

The function $\theta(\mu\cdot\alpha,\mu\cdot\beta)$ can not be chosen arbitrary because of the anti-commutativity of the bracket and of the Jacobi identity.
We observe from (\ref{bracketsolVir}) that if we put: $$e'_{\mu\cdot0}=e_{\mu\cdot0},e'_{\mu\cdot\alpha}=e_{\mu\cdot\alpha}+\frac{\theta(0,\mu\cdot\alpha)}{(\mu\cdot\alpha)}c_\mu,~~ (\alpha\neq\overrightarrow{0}),$$
then we will have $$[e'_{\mu\cdot0},e'_{\mu\cdot\alpha}]_{SVir}=(\mu\cdot\alpha)e'_{\mu\cdot\alpha},~~(\mu\cdot\alpha\in\Gamma_\mu).$$
This transformation is merely a change of basis and we can drop the prime and say that:
\begin{equation}\label{ident2}[e_{\mu\cdot0},e_{\mu\cdot\alpha}]_{SVir}=(\mu\cdot\alpha)e_{\mu\cdot\alpha},~~(\mu\cdot\alpha\in\Gamma_\mu)\end{equation}
From the Jacobi identity for $e_{\mu\cdot0},e_{\mu\cdot\alpha},e_{\mu\cdot\beta}$ we get
\begin{equation}\label{ident3}[e_{\mu.0},[e_{\mu\cdot\beta},e_{\mu\cdot\alpha}]_{SVir}]_{SVir}=\mu\cdot(\beta+\alpha)[e_{\mu\cdot\beta},e_{\mu\cdot\alpha}]_{SVir}\end{equation}
Substituting (\ref{bracketsolVir}) in (\ref{ident3}) and using (\ref{ident2}) we get: $$\mu.(\alpha+\beta)\theta(\mu\cdot\alpha,\mu\cdot\beta)c_\mu=0.$$
But this is equivalent to $\alpha+\beta =\overrightarrow{0}$ or $\theta(\mu\cdot\alpha,\mu\cdot\beta)=0$. Then $\theta$ has the following form: \begin{equation} \label{exptheta}\theta(\mu\cdot\alpha,\mu\cdot\beta)=\delta_{\alpha,-\beta}\eta(\mu\cdot\alpha)\end{equation}
 where $\eta$ is a function from $\Gamma_{\mu}$ to $\mathbb{C}$.

The Lie bracket  (\ref{bracketsolVir}) becomes:
\begin{equation}\label{ident4}[e_{\mu\cdot\alpha},e_{\mu\cdot\beta}]_{SVir}=(\mu.\beta-\mu\cdot\alpha)e_{\mu\cdot(\alpha+\beta)}+\delta_{\alpha,-\beta}\eta(\mu\cdot\alpha)c_\mu,~\mu\cdot\alpha,\mu\cdot\beta\in~\Gamma_\mu\end{equation}
By antisymmetry of the bracket, we deduce that $\eta$ is an odd function ($\eta(\mu\cdot\alpha)=-\eta(-\mu\cdot\alpha)$)and by bi-linearity of the bracket, we deduce that $\eta$ is additive and so, $\eta:(\Gamma_{\mu},+)\rightarrow (\mathbb{C},+)$ is a group morphism.

We now work out the Jacobi identity for $e_{\mu\cdot\kappa},e_{\mu\cdot\alpha},e_{\mu\cdot\beta}$ with $\kappa+\beta+\alpha=\overrightarrow{0}$. Using (\ref{ident4}) and the the fact that $\eta$ is odd, we get:
\begin{equation}\label{ident5}\mu\cdot(\alpha-\beta)\eta(\mu\cdot(\alpha+\beta))-\mu\cdot(2\beta+\alpha)\eta(\mu\cdot\alpha)+\mu\cdot(\beta+2\alpha)\eta(\mu\cdot\beta)=0\end{equation}
where $\eta$ is a continuous function.
Substituting $\beta$ by $-\beta$ in (\ref{ident5}) we obtain the following equation:
\begin{equation}\label{ident6}\mu.(\alpha+\beta)\eta(\mu\cdot(\alpha-\beta))-\mu\cdot(\alpha-2\beta)\eta(\mu\cdot\alpha)-\mu\cdot(2\alpha-\beta)\eta(\mu\cdot\beta)=0\end{equation}
by adding (\ref{ident5}) and (\ref{ident6}) we get:
\begin{equation}\label{ident8}(\mu\cdot\alpha)[\eta(\mu\cdot(\alpha+\beta))+\eta(\mu\cdot(\alpha-\beta))-2\eta(\mu\cdot\alpha)]=(\mu\cdot\beta)[\eta(\mu\cdot(\alpha+\beta))+\eta(\mu\cdot(\beta-\alpha))-2\eta(\mu\cdot\beta)]\end{equation}
Let us denoted $x:=\mu\cdot\alpha$ and $y:=\mu\cdot\beta$ and replace them in (\ref{ident8}) we will obtain:
\begin{equation}\label{ident9}x[\eta(x+y)+\eta(x-y)-2\eta(x)]=y[\eta(x+y)-\eta(x-y)-2\eta(y)].\end{equation}
But (\ref{ident9}) is equivalent to the following equation:
\begin{equation}\label{ident10}2x\eta(x)-2y\eta(y)=(x-y)\eta(x+y)+(x+y)\eta(x-y).\end{equation}

Using results on functional equations by PL.Kannappan,T.Riedel and P.K.Sahoo (see \cite{KaRiSa}), the equation (\ref{ident10})
has the following general solution: $$\eta(x)=ax^{3}+A(x)$$ where $A:\mathbb{C}\mapsto\mathbb{C}$ is an additive function. Since we work with continuous function $\eta$, then $A$ will be continuous and additive function, and so it is a linear function $A(x)= bx, b\in \mathbb{C}$.

 Finally, $\eta(x)=ax^{3}+bx$ where $a,b~\in~\mathbb{C}$ and for $x=\mu\cdot\alpha$ we have: $$\eta(\mu\cdot\alpha)= a(\mu\cdot\alpha)^{3}+b(\mu\cdot\alpha) .$$
The central extension of $\mathbf{W}(n)_{\mu}$ is nontrivial if and only if  $a\neq0$ while $b$ can be chosen arbitrary. By the convention taken in Virasoro $2$-cocycle ( $n=1$ ), the choice $a=-b=\frac{1}{12}$ and the generating $2$-cocycle becomes:
\begin{equation}\label{SVirC}C_{\mu}(e_{\mu\cdot\alpha},e_{\mu\cdot\beta})=\delta_{\alpha,-\beta}\eta(\mu\cdot\alpha)c_\mu=\frac{(\mu\cdot\alpha)^{3}-\mu\cdot\alpha}{12}\delta_{\alpha,-\beta}c_\mu\end{equation}
\end{proof}

Take $y=2x$ in equation \ref{ident10} of the proof, we obtain the following equation:
\begin{equation}\label{identi11}5x\eta(x)-4x\eta(2x)+x\eta(3x)=0\end{equation}
\begin{pro}
A polynomial function $\eta$ is solution of equation (\ref{identi11}) if and only if $$\eta(x)= ax^3+bx.$$
\end{pro}

\begin{proof}

Let $\eta(x)=\displaystyle\sum_{k=0}^{n}a_kx^{k}$ and
substituting $\eta(x)$ in \ref{identi11} we get
$$\displaystyle\sum_{k=0}^{n}a_k(5-2^{k+2}+3^{k})x^{k+1}=0.$$
This is equivalent to $$a_k(5-2^{k+2}+3^{k})=0 \hbox{ for all }  k\in \mathbb{Z}_+~.$$
But $a_k(5-2^{k+2}+3^{k})=0$ implies that $a_k=0$ or $5-2^{k+2}+3^{k}=0$
Now, if $k=0,~2$ or $k\geq4$ then $5-2^{k+2}+3^{k}\neq0,$ and so $a_k=0$ and if $k=1,~3$ then $5-2^{k+2}+3^{k}=0$ and $a_1$ and $a_3$ are arbitrary.
We obtain finally $\eta(x)=a_1x+a_3x^{3}$.
\end{proof}

The following proposition is a consequence of Theorem \ref{Ext1}.
\begin{pro} The second cohomology group of the solenoidal Witt algebra $\mathbf{W}(n)_{\mu}$ with coefficients in the trivial module $\mathbb{C}$ is one dimensional :
$$H^2(\mathbf{W}(n)_{\mu})=\mathbb{C}.$$
\end{pro}

\begin{rmk}
 Let the map $f:\mathbf{W}(n)_{\mu}\rightarrow \mathbb{C}$ defined by $f(e_{\mu\cdot\alpha})=\frac{1}{2}\delta_{\mu\cdot\alpha,0}$. Then
  $f([e_{\mu\cdot\alpha},e_{\mu\cdot\beta}])=(\mu.\beta)\delta_{\mu\cdot(\alpha+\beta),0}$ and
  $\widetilde{f}:\mathbf{W}(n)_{\mu}\times \mathbf{W}(n)_{\mu}\rightarrow \mathbb{C}$ defined by $$\widetilde{f}(e_{\mu\cdot\alpha},e_{\mu\cdot\beta})=(\mu\cdot\beta)\delta_{\mu\cdot(\alpha+\beta),0}$$
  is a 2-coboundary. So, the multiplication rule (for $\mathbf{Vir}(n)_{\mu}$) can be replaced by any rule:
$$[e_{\mu\cdot\alpha},e_{\mu\cdot\beta}]=\mu\cdot(\beta-\alpha)e_{\mu\cdot(\alpha+\beta)}+\big(a(\mu\cdot\alpha)^3+b\mu\cdot\alpha\big)\delta_{\mu\cdot\alpha,-\mu\cdot\beta}c_\mu$$
with constants $a,b\in\mathbb{C}$ and $a\neq 0$.
It is customary to use the normalization
$$\frac{(\mu\cdot\alpha)^3-\mu\cdot\alpha}{12}.$$
\end{rmk}

\section{Harish-Chandra modules over $\mathbf{Vir}(n)_\mu$}
\subsection{Generalities on Harish-Chandra modules}
\begin{defi}  A $\mathbf{Vir}(n)_\mu$-module $V$ is said to be a Harish-Chandra module or admissible module if :\\
\begin{itemize}
\item[1)] $d_\mu$ acts semisimply on $V$. \\
\item[2)] the eigenspaces of $d_\mu$ are finite-dimensional.\\
\end{itemize}
The eigenspaces of $d_\mu$ will be also called the weight spaces of $V$.
\end{defi}
\begin{defi} An admissible $\mathbf{Vir}(n)_\mu$-module $V$ is said to be cuspidal or bounded if the dimensions
of weight spaces are uniformly bounded by a constant.
\end{defi}
\begin{defi}. A $\mathbf{Vir}(n)_\mu$-module $V$ is indecomposable if a decomposition of $V$ in a
direct sum of non-trivial $\mathbf{Vir}(n)_\mu$-submodules does not exist.
\end{defi}
As a direct consequence of the commutation relations $[d_\mu,e_{\mu\cdot\alpha}]=(\mu\cdot\alpha)e_{\mu\cdot\alpha}$ is the existence of a decomposition of Harish-Chandra indecomposable $\mathbf{Vir}_\mu(n)$-module $V$ to weight subspaces: $$V=\oplus_{\alpha\in \mathbb{Z}^n} V_{a+\mu\cdot\alpha}$$
where $ V_{a+\mu\cdot\alpha}:=\{v\in V|d_\mu.v=(a+\mu\cdot\alpha)v\}$ and $a$ is a complex number. Then the support of $V$ is a subset of the coset $a+\Gamma_{\mu}$.
So any Harish-Chandra $\mathbf{Vir}(n)_\mu$-module can be decomposed into a direct sum of submodules
corresponding to distinct cosets of $\Gamma_\mu$ in $\mathbb{C}$. So we can limited our study to modules with weights in a coset $a+\Gamma_\mu, a\in \mathbb{C}$. In the following, the supports of considered modules will be in a coset $a+\Gamma_\mu$.
The following result gives an important and elementary property of the action of the central element $c_\mu$ on an admissible $\mathbf{Vir}(n)_\mu$-module.
\begin{pro}
In an irreducible  Harish-Chandra  $\mathbf{Vir}(n)_\mu$-module, $c_\mu$ acts by a scalar.\\
\end{pro}
\begin{proof}
In both cases, $c_\mu$ and $d_\mu$ commutes so that $V_{a+\mu.\alpha}$ is invariant by $c_\mu$. Thus, the $\mathbf{Vir}(n)_\mu$-module
$V$ can be decomposed into a direct sum, at most countable sum of characteristic
subspaces of $c_\mu$, each of them being invariant by $\mathbf{Vir}(n)_\mu$. In case of an indecomposable $\mathbf{Vir}(n)_\mu$-module $V$,
 $c_\mu$ has at most one characteristic subspace. In case of an irreducible $\mathbf{Vir}(n)_\mu$-module $V$
 the eigenspaces of $c_\mu$ are invariant by $\mathbf{Vir}(n)_\mu$ and the proposition is proved.
\end{proof}

\subsection{Classification of cuspidal modules over $\mathbf{Vir}(n)_\mu$}

 Let $A_n=\mathbb{C}[t_{i}^{\pm1},~1\leq i\leq n]$ be the algebra of Laurent polynomials and let $v_{\mu\cdot\beta}:= t^{\mu\cdot\beta}$. We define a two parameter action on $A_n$ by:
 \begin{equation}\label{actCusp} e_{\mu\cdot\alpha}v_{\mu\cdot\beta} = (\mu\cdot\beta+a+(\mu\cdot\alpha)b)v_{\mu\cdot(\alpha+\beta)}
 \hbox{ where }~\mu\cdot\alpha,\mu\cdot\beta\in\Gamma_\mu  \hbox{ and } a,b\in \mathbb{C}.\end{equation}
 We denote such $\mathbf{W}(n)_{\mu}$-module by $T_{\mu}(a,b)$.
 These modules are no other then the  $\mathbf{W}(n)_{\mu}$-modules of tensor fields studied by Y. Billig and V. Futorny in \cite{BiFu}. They are
cuspidal modules since every weight space in $T_{\mu}(a,b)$ is 1-dimensional.

Two $\mathbf{W}(n)_{\mu}$-modules $T_{\mu}(a,b)$ and $T_{\mu}(a',b')$  are isomorphic if and only if $a-a'\in\Gamma_{\mu}$. The isomorphism is given by
$\varrho:T_{\mu}(a,b)\rightarrow T_{\mu}(a',b');~v_{\mu\cdot\beta}\mapsto v_{\mu\cdot\beta+a-a'}$. So to each coset in $\mathbb{C}/\Gamma_{\mu}$ is associated one and only one  class of isomorphisms of such modules. In particular, to the coset $\Gamma_{\mu}$ is associated the class of the module $T_{\mu}(0,b)$.

The algebraic dual of the $\mathbf{W}(n)_{\mu}$-module $T_{\mu}(a,b)$ is again a tensor field module and is isomorphic to $T_{\mu}(-a,1-b)$.

The irreducibility of the modules  $T_{\mu}(a,b)$ is given by the following proposition:

\begin{pro}

\begin{itemize}
 \item [1)] $T_{\mu}(a,b)$ is irreducible unless $b\in\{0,1\}$ and $a\in\Gamma_{\mu}.$
 \item [2)] If $b\in\{0,1\}$ and $a\in\Gamma_{\mu}.$ It suffices to consider $(a,b)=(0,0)$ or $(a,b)=(0,1)$.\\
 i) If $(a,b)=(0,0),~T_{\mu}(0,0)$ contain the trivial sub-module $\mathbb{C}v_0$ and the quotient $T_{\mu}(0,0)/\mathbb{C}v_0$ is irreducible.\\
 ii) If $(a,b)=(0,1),~ T_{\mu}(0,1)$ contains an irreducible sub-module  $T'_{\mu}(0,1)$ of codimension $1$.
\end{itemize}
\end{pro}

The following theorem is due to Y. Billig and V. Futorny  Theorem 4.8  \cite{BiFu} see also Y. Su, K.Zhao \cite{SuZhao} Theorem 4.6.
\begin{thm}\label{CuspWitt}
Every simple cuspidal $\mathbf{W}(n)_{\mu}$-module is a sub-quotient of a module of intermediate series and it is isomorphic to either:
\begin{itemize}
\item[1)]A module of tensor fields $T_\mu(a, b)$ where $a\not\in\Gamma_\mu$  $\hbox{or}~b\neq0,1$.
\item[2)]The quotient $\overline{T}_\mu(0,0)=T_\mu(0,0)/\mathbb{C}v_0$ of the module $A_n$ of Laurent polynomials for the natural action by the $1$-dimensional sub-module $\mathbb{C}v_0$ of constants.
\item[3)]The trivial $1$-dimensional module.
\end{itemize}
\end{thm}
\begin{rmk}
By duality, the tensor module $T_{\mu}(0,1)$ will be reducible and $T'_{\mu}(0,1):=\oplus_{s\in \Gamma_\mu\setminus\{0\}}\mathbb{C}v_s$ is a submodule of codimension one isomorphic to $\overline{T}_\mu(0,0)$.
\end{rmk}

 The following theorem reduces the classifications of simple cuspidal $\mathbf{Vir}(n)_\mu$-modules to the classification of simple  $\mathbf{W}(n)_\mu$-cuspidal modules given by Theorem \ref{CuspWitt}.

\begin{thm}
 In an irreducible cuspidal  $\mathbf{Vir}(n)_\mu$-module $V$, $c_\mu$ acts trivially.
\end{thm}
\begin{proof}
To prove that $c_{\mu}$ acts by zero, consider the subalgebra $\mathbf{Vir}_i$ of $\mathbf{Vir}(n)_\mu$  given by:  $$\mathbf{Vir}_i:=\oplus_{n\in \mathbb{Z}}\mathbb{C}t_i^nd_\mu\oplus \mathbb{C}c_\mu.$$
This algebra is  isomorphic to $\mathbf{Vir}$ where the isomorphism is given by:
$$\varphi_i:\mathbf{Vir}\rightarrow \mathbf{Vir}_i,~e_m\mapsto e_m^i=\mu_i^{-1}t_i^md_\mu,~C\mapsto c_\mu.$$
By the restriction of the action of $\mathbf{Vir}(n)_\mu$ on $V$ to $\mathbf{Vir}_i$ , the module $V$ becomes a cuspidal module of $\mathbf{Vir}_i$. Then by the isomorphism $\varphi_i$ it becomes a cuspidal $\mathbf{Vir}$-module. Now  we apply Theorem II.7 of Martin and  Piard \cite{MP} or Proposition 5.9 of Chary and Pressley \cite{ChPr}, the central charge acts by zero on a cuspidal $\mathbf{Vir}$-modules. Then $c_\mu$ acts trivially on $V$.
\end{proof}

\subsection{Generalized highest weight modules over $\mathbf{Vir}(n)_{\mu}$}
For $\mu=(\mu_1,\mu_2,\ldots,\mu_n)\in \mathbb{C}^n$, let $\mu'=(\mu_2,\ldots,\mu_n)\in \mathbb{C}^{n-1}$. For any $\alpha=(\alpha_1,\ldots,\alpha_n)\in \mathbb{Z}^{n}$ we have $\mu\cdot\alpha=\mu_1\alpha_1+\mu'\c\alpha'$  where $\alpha'=(\alpha_2,\ldots,\alpha_n)$. This induces a natural embedding of $\Gamma_{\mu'}$ in $\Gamma_{\mu}$
given by $\mu'\c\alpha'\mapsto \mu\cdot(0,\alpha')$. The embedding $\Gamma_{\mu'}\hookrightarrow \Gamma_{\mu}$ as defined below, induces an embedding of the Lie algebra $\mathbf{HVir}(n-1)_{\mu'}$ into the Lie algebra $\mathbf{HVir}(n)_{\mu}$ given by:
$$e_{\mu'\c\alpha'}\mapsto e_{\mu\cdot(0,\alpha')}$$
Now we assume that $\Gamma_{\mu} =\mathbb{ Z}\mu_{1} \oplus \Gamma_{\mu'} \subset\mathbb{ C} \hbox{ where } 0
\neq \mu_1 \in \mathbb{ C }$ and $\Gamma_{\mu'}$ is a nonzero subgroup
of $\mathbb{C}$.
Let $A_{n-1}=\mathbb{C}[t_2^{\pm1},\ldots,t_{n}^{\pm1}]$, then we have the following $\mathbb{Z}$-grading of $\mathbf{Vir}(n)_{\mu}$:
$$\mathbf{Vir}(n)_{\mu}=\oplus_{i\in\mathbb{Z}}\mathbf{Vir}(n)_{\mu}^{i}$$ where $\mathbf{Vir}(n)_{\mu}^{i}=t_1^{i}A_{n-1}d_{\mu}\oplus\mathbb{C}c_\mu$ if $i\neq 0$ and $\mathbf{Vir}(n)_{\mu}^{0}=A_{n-1}d_{\mu}\oplus\mathbb{C}c_{\mu}.$
Then, $\mathbf{Vir}(n)_{\mu}^{0}$ is isomorphic to $\mathbf{Vir}(n-1)_{\mu'}$.

For $a,b\in\mathbb{C}$, we denote $T_{\mu'}(a,b)$ the $\mathbf{Vir}(n)_{\mu}^{0}$ module of tensor fields $$T_{\mu'}(a,b)=\oplus_{\mu'\c\kappa'\in\Gamma_{\mu'}}\mathbb{C}v_{\mu'\c\kappa'}$$
subject to the action :
\begin{equation}\label{act1}e_{\mu'\c\alpha'}.v_{\mu'\c\kappa'}=(a+\mu'\c\kappa'+b(\mu'\c\alpha'))v_{\mu'\c(\alpha'+\kappa')},c_{\mu}.v_{\mu'\c\kappa'}=0 \hbox{ for }\mu'\c\alpha',\mu'\c\kappa'\in\Gamma_{\mu'}.\end{equation}
Let $\mathbf{Vir}(n)_{\mu}^{\pm}:=\oplus_{i\in\pm\mathbb{N}}\mathbf{Vir}(n)_{\mu}^{i}$.
We extend the $\mathbf{Vir}(n)_{\mu}^{0}$ module structure on $T_{\mu'}(a,b)$ given by (\ref{act1}) to $\mathbf{Vir}(n)_{\mu}^{+}\oplus\mathbf{Vir}(n)_{\mu}^{0}$
where the elements of $\mathbf{Vir}(n)_{\mu}^{+}$ act by zero on $T_{\mu'}(a,b)$.
Then by induction, we obtain the generalized Verma module
$$\widetilde{V}(a,b,\Gamma_{\mu'})=Ind^{\mathbf{Vir}(n)_{\mu}}_{\mathbf{Vir}(n)_{\mu}^{+}\oplus\mathbf{Vir}(n)_{\mu}^{0}}T_{\mu'}(a,b).$$
As vector spaces we have $\widetilde{V}(a,b,\Gamma_{\mu'})\cong U(\mathbf{Vir}(n)_{\mu}^{-})\otimes_{\mathbb{c}}T_{\mu'}(a,b).$
 The module $\widetilde{V}(a,b,\Gamma_{\mu'})$ has a unique maximal proper
submodule $J(a,b,\Gamma_{\mu'})$ trivially intersecting  $T_{\mu'}(a,b)$. The quotient module
$$\overline{V}=\overline{V}(a,b,\Gamma_{\mu'}):=\widetilde{V}(a,b,\Gamma_{\mu'})/J(a,b,\Gamma_{\mu'})$$
 is uniquely determined by the constants $a,b$ and $$\overline{V}(a,b,\Gamma_{\mu'})=\oplus_{i>0} \overline{V}_{a-i\mu_1+\Gamma_{\mu'}}$$
where
$\overline{V}_{a-i\mu_1+\Gamma_{\mu'}} = \oplus_{\mu'\c\kappa\in\Gamma_{\mu'}}\overline{V}_{a-i\mu_1+\mu'\c\kappa}$
and $$\overline{V}_{a-i\mu_1+\mu'\c\kappa}=\{v\in\overline{V}/d_{\mu}v=(a-i\mu_1+\mu'\c\kappa)v\}$$
We can similarly define $\widetilde{V}_{a+i\mu_1+\Gamma_{\mu'}}$ and $\widetilde{V}_{a-i\mu_1+\Gamma_{\mu'}}$.

The following theorem is due to S. Berman and Y. Billig \cite{BeBill}, Theorem 1.12. See also Y. Billig and K. Zhao \cite{BiZa}, Theorem 1.5 and generally Theorem 3.1, also Theorem 2.4 in R. Lu and K. Zhao \cite{LuZhao}.

\begin{thm}
The weight spaces of the $\mathbf{Vir}(n)_{\mu}$-module $\overline{V}(a,b,\Gamma_{\mu'})$ are finite dimensional.
More precisely, $dim\overline{V}_{a-i\mu_1+\mu'\c\kappa}< 1.3 .\ldots. (2i + 1) \hbox{ for all } i \in \mathbb{N},\mu'\c\kappa\in\Gamma_{\mu'}.$
\end{thm}
Let $V=\oplus_{u\in \Gamma_{\mu}} V_{a+u},~~a\in \mathbb{C}$ be an irreducible weight module with finite dimensional weight spaces ($dim V_{a+u}<\infty$).
The following definition of generalized highest weight modules (\textbf{GHW} module for shirt) is due to Y. Su \cite{Su} see also \cite{SuZhao}.
\begin{defi} Let $(u_1,\ldots,u_n)$ be a $\mathbb{Z}$-basis of $\Gamma_\mu$ and let $\Gamma_\mu^{>0}:= \mathbb{Z}^+u_1\oplus\ldots\oplus\mathbb{Z}^+u_n$ and $\mathbf{Vir}(n)_{\mu}^{>0}:= \oplus_{u\in \Gamma_\mu^{>0}} (\mathbf{Vir}(n)_{\mu})_u.$ Assume that there exists $\lambda_0\in Supp(V)$ and nonzero $v_{\lambda_0}\in V_{\lambda_0}$ such that: $\mathbf{Vir}(n)_{\mu}^{>0}v_{\lambda_0}=0$. Then $V$ is said to be a generalized highest weight module
with generalized highest weight $\lambda_0$ and generalized highest weight vector $v_{\lambda_0}$. Such module $V$ is denoted by $V(\lambda_0)$.
\end{defi}

\begin{rmk}
Since for the choice of a $\mathbb{Z}$-basis $(u_1,\ldots,u_n)$ of $\Gamma_\mu$, the $n$-tuple $(-u_1,\ldots,-u_n)$ is also a $\mathbb{Z}$-basis. The notion of generalized highest weight modules and the notion of generalized lowest modules coincide. We will consider only generalized highest weight modules.
\end{rmk}

Applying
 Theorem 3.9. of R. Lu and K. Zhao, \cite{LuZhao} and Theorem 1.2. of Y. Su, \cite{Su}, we have the following theorem in the case of the Lie algebra  $\mathbf{Vir}(n)_{\mu}$.

\begin{thm} \label{thmLZS}
Any irreducible $\mathbf{Vir}(n)_{\mu}$-module $\overline{V}(a,b,\Gamma_{\mu'})$ is a generalized highest weight module $V(\lambda_0)$ for some $\lambda_0$ in $Supp(\overline{V}(a,b,\Gamma_{\mu'}))$.
\end{thm}

In Y. Su, \cite{Su1,Su} , it is proved that for a generalized Virasoro algebra an irreducible weight module with finite dimensional weight spaces is either a cuspidal or a generalized highest weight module. In our particular case we have the following theorem:

\begin{thm}\label{thmS}
  Let $V$ is a nontrivial irreducible weight $\mathbf{Vir}(n)_{\mu}$-module with finite dimensional weight spaces. Then $V$ is one of the following classes:
\begin{itemize}
\item[a)]If $V$ is uniformly bounded, then $V \cong T_{\mu}(a,b)$ for $(a,b)\in \mathbb{C}^{2} \backslash \{(0,0)\}$ or $V\cong \overline{T}_{\mu}(0,0).$
\item[b)]If $V$ is not uniformly bounded, then $V$ is a \textbf{GHW} module.

\end{itemize}
\end{thm}
The following theorem is a consequence of Theorem \ref{thmLZS} and Theorem \ref{thmS} (see R. Lu and K. Zhao \cite{LuZhao}, Theorem 3.9 and X. Guo, R. Lu and K.Zhao \cite{XRK} Theorem 3.9 for the general case).

\begin{thm} \label{ClssThm}
If $V$ is a nontrivial irreducible weight module with finite dimensional
weight spaces over the solenoidal-Virasoro algebra $\mathbf{Vir}(n)_{\mu}$ , then $V$ is isomorphic to one of the following modules:
\begin{itemize}
\item[1)]$V\cong T_{\mu}(a,b)$ for $(a,b)\in \mathbb{C}^2\setminus\{(0,0)\}$ or $V\cong \overline{T}_{\mu}(0,0)$.
\item[2)] $V\cong \overline{V}(a,b,\Gamma_{\mu'})$ for some $a,b\in \mathbb{C}$.
\end{itemize}
\end{thm}

The following theorem is a consequence of Theorem 3.9 in \cite{LuZhao} and Theorem 3.9 in \cite{XRK}.It classifies Harish-chandra modules of $\mathbf{Vir}(n)_{\mu}.$

\begin{thm} \label{GenThm}
Let  $V$ be a nontrivial irreducible weight module with finite dimensional
weight spaces over the solenoidal-Virasoro algebra $\mathbf{Vir}(n)_{\mu}$.
\begin{itemize}
\item[1)] If $n=1$ then $\Gamma_{\mu}=\mu\mathbb{Z}\simeq\mathbb{Z}$, then $V$ is  of intermediate series or highest or lowest module (see \cite{Ma}).
\item[2)] If $n\geq 2$, then $V$ is isomorphic to one of the following modules:
\begin{itemize}
\item[a)]$V\cong T_{\mu}(a,b)$ for $(a,b)\in \mathbb{C}^2\setminus\{(0,0)\}$ or $V\cong \overline{T}_{\mu}(0,0)$.
\item[b)] $V\cong \overline{V}(a,b,\Gamma_{\mu'})$ for some $a,b\in \mathbb{C}$.
\end{itemize}

\end{itemize}
\end{thm}

\section{ $\mathbf{Vir}(n)_\mu$-modules with infinite dimensional weight spaces}\label{highestLowest}
   Let $\displaystyle\mathbb{Z}^{n}$ be the free abelian group of rank $n$ whose elements are sequences of $n$ integers, and operation is the addition.
A group order on $\displaystyle \mathbb{Z}^{n}$ is a total order, which is compatible with addition, that is
$$ a<b\quad {\text{ if and only if }}\quad a+c<b+c.$$
The lexicographical order $<_{lex}$ is a group order on $\mathbb{Z}^{n}$

We transport the lexicographic order $<_{lex}$ on $\mathbb{Z}^{n}$ to $\Gamma_\mu$ that is
$$\mu\cdot\alpha\prec \mu\cdot\beta \hbox{ if and only if } \alpha<_{lex} \beta.$$

Let us introduce
$\Gamma_\mu^+:=\{\mu\cdot\alpha| \overrightarrow{0}<_{lex}\alpha\}$ and $\Gamma_\mu^-:=\{\mu\cdot\alpha| \alpha<_{lex}\overrightarrow{0}\}$
and let us denote:
$$(\mathbf{W}(n)_{\mu})_{\pm}:=\displaystyle \oplus_{\mu\cdot\alpha\in\Gamma_\mu^{\pm} }\mathbb{C}e_{\mu\cdot\alpha}$$
 then we have:
$$\mathbf{W}(n)_{\mu}=(\mathbf{W}(n)_{\mu})_+\oplus(\mathbf{W}(n)_{\mu})_0\oplus (\mathbf{W}(n)_{\mu})_{-}$$
where $(\mathbf{W}(n)_{\mu})_0=\mathbb{C}d_\mu.$

We deduce the following triangular decomposition of  $\mathbf{Vir}(n)_{\mu}$:
$$\mathbf{Vir}(n)_{\mu}=(\mathbf{Vir}(n)_{\mu})_+\oplus(\mathbf{Vir}(n)_{\mu})_0\oplus (\mathbf{Vir}(n)_{\mu})_{-}$$
where, $(\mathbf{Vir}(n)_{\mu})_0=\mathbb{C}d_\mu\oplus \mathbb{C}c_\mu$ is the Cartan subalgebra (CSA) and $(\mathbf{Vir}(n)_{\mu})_{\pm}=(\mathbf{W}(n)_{\mu})_{\pm}$.

The root decomposition of $\mathbf{Vir}(n)_{\mu}$ with respect to the CSA  $(\mathbf{Vir}(n)_{\mu})_0$ is the following:
$$\mathbf{Vir}(n)_{\mu}=\displaystyle\oplus_{\mu\cdot\alpha\in \Gamma_\mu}\mathbf{Vir}(n)_{\mu\cdot\alpha}$$
where $\mathbf{Vir}(n)_{\mu\cdot\alpha}=\mathbb{C}e_{\mu\cdot\alpha}.$ Moreover, we have:
$$(\mathbf{Vir}(n)_{\mu})_{+}=\displaystyle\oplus_{\mu\cdot\alpha\in\Gamma_\mu^+}\mathbf{Vir}(n)_{\mu\cdot\alpha}\hbox{ and } (\mathbf{Vir}(n)_{\mu})_-=\displaystyle\oplus_{\mu\cdot\alpha\in \Gamma_\mu^-}\mathbf{Vir}(n)_{\mu\cdot\alpha}.$$

A $\mathbf{Vir}(n)_{\mu}$-module $V$  is called a weight module, if $V=\oplus_{(\lambda,c)\in\mathbb{C}^2}V_{\lambda,c}$, where: $$V_{\lambda,c} := \{v \in V |d_{\mu}.v = \lambda v, c_\mu.v =cv\}.$$ A $(\lambda,c)$ is called a weight of $V$ if $V_{\lambda,c}\neq \{0\}$.
When $C$ acts as a scalar $c$ on the whole module $V$ , we shall simply write
$V_{\lambda}$ instead of $V_{\lambda,c}$.

In the rest of this section, all modules considered are such modules.
A $\mathbf{Vir}(n)_\mu$-module V is called a weight module if V is the sum of its weight spaces
. For a weight module $V$ , we define $suppV := \{\lambda\in \mathbb{C} \backslash V_{\lambda}\neq 0\}$, which is
generally called the weight set (or the support) of $V$ . Given a weight module $V$, we denote $V =\bigoplus_{\lambda\in\mathbb{C}}V_\lambda$, where
$V_\lambda = 0$ for $\lambda~\not\in~suppV.$
Let $V$ be a module and $W'\subset  W$ are submodules of $V$ . The module $W/W'$
is called a sub-quotient of $V$ . If $W' = 0$ we consider that $W = W/W'$.

Let $V$ be a weight module over $\mathbf{Vir}(n)_\mu$. A vector $v \in V_{\lambda,c}, \lambda \in suppV , c \in \mathbb{C}$, is called a
highest weight  (resp. lowest weight) vector if $(\mathbf{Vir}(n)_\mu)_{+}v=0$ (resp. $(\mathbf{Vir}(n)_\mu)_{-}v=0$).
$V$ is called a highest weight (resp. lowest weight) module with highest weight (resp.
lowest weight) $(\lambda, c)$ if there exists a nonzero highest (lowest, resp.) weight vector $v \in V_{\lambda,c}$
such that $V$ is generated by $v$.
Let $\lambda\in \Gamma_\mu, c\in \mathbb{C}$. Let us consider a one dimensional module $\mathbb{C}_{\lambda,c}$ over $\mathfrak{b}^+:=(\mathbf{Vir}(n)_\mu)_{0}\oplus(\mathbf{Vir}(n)_\mu)_{+}$ such
that $(\mathbf{Vir}(n)_\mu)_{+}$ acts by zero, $d_\mu$ is a multiplication by $\lambda$ and $C$ is a multiplication by
$c$. Verma module $M(\lambda,c)$ over the solenoidal Virasoro algebra is by definition an induced
module from $\mathbb{C}_{\lambda,c}$
$$M(\lambda,c) = Ind^{\mathbf{Vir}(n)_\mu}_{\mathfrak{b}^+}\mathbb{C}_{\lambda,c}.$$
We have a natural inclusion of $\mathbb{C}_{\lambda,c}\hookrightarrow M(\lambda,c).$ So we have a vector $v\in M(\lambda,c)$
corresponding to $1\in \mathbb{C}_{\lambda,c}$ . Sometimes we will write $v_{\lambda,c}$ to stress that this
vector lies in $M(\lambda,c).$ The vector $v$ is called the vacuum vector.
The enveloping algebra $U((\mathbf{Vir}(n)_{\mu})_-)$ is graded by its weight subspaces $U((\mathbf{Vir}(n)_{\mu})_-)_{\mu\cdot\alpha}$, that is $$U((\mathbf{Vir}(n)_{\mu})_-)=\oplus_{\alpha\in \mathbb{Z}^n,\alpha<_{lex} \overrightarrow{0}}U((\mathbf{Vir}(n)_{\mu})_-)_{\mu\cdot\alpha}.$$

Let us make a few remarks about Verma modules. First, any Verma
module $M(\lambda,c)$ is a free module over $U((\mathbf{Vir}(n)_\mu)_{-})$. Therefore, we have the following
basis in $M(\lambda,c)$:
$$e_{\mu\cdot\alpha_k}e_{\mu\cdot\alpha_{k-1}}\ldots e_{\mu\cdot\alpha_2}e_{\mu\cdot\alpha_1} v_{\lambda,c}$$
where $\mu\cdot\alpha_i\in \Gamma_\mu^-$ for all $i=1,\ldots,k$.
The operator $d_\mu$ acts semi-simply on  $M(\lambda,c)$. We can consider the eigenspace
decomposition of  $M(\lambda,c)$,
$$M(\lambda,c) =\displaystyle\oplus_{\mu\cdot\alpha\in \Gamma_\mu^-\cup\{\overrightarrow{0}\}}M(\lambda,c)_{\lambda+\mu\cdot\alpha}.$$
It is easy to see that this
decomposition respects the grading on $\mathbf{Vir}(n)_\mu$.
 We say that vector $w\in M(\lambda,c)$ has level $\mu\cdot\alpha\in \Gamma_\mu^+$ if $w\in M(\lambda,c)_{\lambda-\mu\cdot\alpha}$. A vector $w$ is called singular if it has some level $\mu.\alpha\in \Gamma_\mu^+$ and $(\mathbf{Vir}(n)_{\mu})_+$ acts
by zero on this vector. It is obvious that any singular vector generates a
submodule isomorphic to Verma module. If a singular vector has level $\mu\cdot\alpha\in\Gamma_\mu^+$
then it generates $M(\lambda-\mu\cdot\alpha,c)$.
Let $\overline{M(\lambda,c)}$ be the maximal proper submodule of $M(\lambda,c)$. Then the quotient $$V(\lambda,c):=M(\lambda,c)/\overline{M(\lambda,c)}$$
is an irreducible highest weight module of $\mathbf{Vir}(n)_{\mu}$ and every highest irreducible module is constructed with this manner.

We can construct irreducible lowest weight modules as follows. We consider trivial modules  $\mathbb{C}_{\lambda,c}$ over
$\mathfrak{b}^-:=(\mathbf{Vir}(n)_\mu)_{0}\oplus(\mathbf{Vir}(n)_\mu)_{-}$ where elements of $(\mathbf{Vir}(n)_\mu)_{-}$ act trivially. Then we consider the induced module
$$M(\lambda,c)^{\vee} =Ind^{\mathbf{Vir}(n)_\mu}_{\mathfrak{b}^{-}}\mathbb{C}_{\lambda,c}.$$
This module has an irreducible quotient denoted by $V(\lambda,c)^{\vee}$ and which is an irreducible lowest module. Moreover every irreducible lowest weight
module is constructed with this manner.
\begin{thm}
Let $V(\lambda,c)$  be the  irreducible highest weight module of $\mathbf{Vir}(n)_\mu$ constructed above, then there exists $\alpha\in supp(V(\lambda,c))$ such that $V(\lambda,c)_\alpha$ is an infinite dimensional  weight subspace of $V(\lambda,c)$.

We have the same assertion for the lowest weight module $V(\lambda,c)^{\vee}$.
\end{thm}
\begin{proof}To prove this theorem, we consider first the case $\Gamma_{\mu} \simeq\mathbb{Z}^2$.
The statement of the theorem is a consequence of Corollary 3.4 in \cite{LuZhao}, where R. Lu and K. Zhao proved it for generalized Virasoro algebra in general.
Now, if  $\Gamma_{\mu} \simeq\mathbb{Z}^n,n\geq 2$. Assume that $dim(V(\lambda,c))_{\kappa}$ finite dimensional for any $\kappa\in supp(V(\lambda,c))$. By its construction $supp(V(\lambda,c))\subset \lambda-\Gamma_{\mu}^+$. But  $\Gamma_{\mu}^+=\mu_1\mathbb{N}+\Gamma_{\mu'}$.

Let $B_{\varepsilon}=(\varepsilon_1,\ldots,\varepsilon_n)$ be the canonical basis of $\mathbb{Z}^n$ considered as $\mathbb{Z}$-module and let $B=(e_1,\ldots,e_n)$ the basis of $\Gamma_{\mu}$ image of the basis $B_{\varepsilon}$ by the isomorphism $\alpha\mapsto \mu\cdot\alpha$.
Then $\Gamma_{\mu}=\mathbb{Z}e_1\oplus \Gamma_{\mu'}$.

We have $\mu_1\mathbb{N}+\Gamma_{\mu'}\cap supp(V(\lambda,c))=\emptyset$ and $\lambda+\Gamma_{\mu'}\cap supp(V(\lambda,c))\neq \emptyset$.

Using the same argument as in the proof of Claim 2 of Lemma 3.3 in \cite{LuZhao}, the space
$$W = \oplus_{\kappa\in \lambda+\Gamma_{\mu'}}V(\lambda,c)_{\kappa}$$
is a uniformly bounded $\mathbf{Vir}(n-1)_{\mu'}$-module and $V(\lambda,c)$ will be a \textbf{GHM} by Theorem \ref{ClssThm}.
But the module $W$ contains the submodule $U(\mathbf{Vir}(n-1)_{\mu'}).v_{\lambda,c}$ which is a highest weight module with highest weight $(\lambda,c)$ which
is not uniformly bounded. A contradiction. Hence $V(\lambda,c)$ contains an infinite dimensional weight subspace.
\end{proof}

\end{document}